\documentclass[12pt]{amsart}
\usepackage{amssymb}
\usepackage{epsfig}

\usepackage[all]{xy}

\theoremstyle{plain}
\newtheorem{thm}{Theorem}[section]
\newtheorem{lmm}[thm]{Lemma}
\newtheorem{cor}[thm]{Corollary}
\newtheorem{prop}[thm]{Proposition}
\theoremstyle{definition}
\newtheorem{dfn}[thm]{Definition}
\newtheorem{rmk}[thm]{Remark}

\def\al{\alpha}

\def\de{\delta}
\def\De{\Delta}
\def\ep{\epsilon}

\def\si{\sigma}
\def\Si{\Sigma}
\def\vp{\varphi}
\def\om{\omega}
\def\ka{\kappa}
\def\op{\operatorname}
\def\ov{\overline}

\def\sm{\setminus}

\def\Cal{\mathcal}
\def\Bbb{\mathbb}
\def\op{\operatorname}
\newenvironment{pf}{\begin{proof}}{\end{proof}}
\begin{document}
\title{Infinitary braid groups}

\author{Katsuya Eda}
\address{School of Science and Engineering, 
Waseda University, Tokyo 169, JAPAN}
\email{eda@waseda.jp}

\author{Takeshi Kaneto}
\address{Institute of Mathematics, 
University of Tsukuba, Tsukuba 305, JAPAN}

\thanks{The authors thank J. Morita, A. Tsuboi and K. Sakai for 
stimulating talks.}
\maketitle
This preprint was written before 1993 when the first author was 
in Uviversity of Tsukuba. Now, according to the new result by 
W. Herfort and W. Hojka the conclusion of Lemma 2.11 and hence 
that of Theorem 2.8 becomes ``cotorsion" instead of ``complete mod-U''. 

\section{Introduction and definitions}
An infinitary version of braid groups has been implicitly 
considered as a direct limit of $n$-braid groups (see Remark~\ref{rmk:limit}).
However, as an intuitive 
object, we can imagine a more complicated braid with infinitely many 
strings (see Figure ($\ast$)). 
%

In this paper we introduce an infinitary version of braid groups and 
study about fundamental properties of it especially in case the number of 
strings is countable. 
Our notation and notion are usual ones for braid groups and topology 
[\bf 2\rm , \bf 8\rm , \bf 11\rm ]. 
$\Bbb Z, \Bbb R$ and $\Bbb C$ are the set of integers, real numbers and 
complex numbers, respectively. The unit interval $\Bbb I$ is $[0,1].$ 
For a path $f$ in a space $X,$ i.e. a continuous map $f:\Bbb I\to X,$ $f^-$ 
is the path defined by $f^-(t)=f(1-t).$ 
For paths $f,g:\Bbb I\to X$ with $f(1) = g(0),$ $f\cdot g$ is the path 
defined by:  $f\cdot g(t) = f(2t)$ for $0\le t\le 1/2$ and $f\cdot g(t)=g(2t-1)$ 
for $1/2\le t\le 1.$ 
For paths $f$ and $g,$ 
$f\sim g$ if $f$ and $g$ are homotopic relative to $\{ 0, 1\}.$ 
A constant path is denoted by the point as usual. 

For a subset $X$ of $\Bbb C$ let $\Bbb C^X$ be the product space, regarding $X$ as an indexed set. For distinct $x,y\in X,$ let 
$\De _{xy} = \{u\in \Bbb C^X: u(x) = u(y)\}$ and 
$\De _X = \bigcup\{ \De _{xy}: x\neq y, x,y\in X\}.$ 
The subspace $\Bbb C^X\sm \De _X$ of $\Bbb C^X$ is denoted by $F_X.$ 
For a path $f$ in $F_X$ and $Y\subset X \subset \Bbb C,$ 
the restriction $f|_Y$ is a path in $F_Y,$ 
where $f|_Y(t)(y) = f(t)(y)$ for $y\in Y.$ 
Let $\Sigma _X$ be the group consisting of all permutations on the discrete 
set $X,$ which also acts freely on $F_X$ as coordinate transformations. 
For a path $f$ in $F_X$ and $\si \in \Si _X,$ $f^{\si }$ is the path in $F_X$ defined by: 
$f^\si (t)(x) = f(t)(\si (x))$ for $t\in \Bbb I, x\in X.$ 
\begin{dfn}\label{dfn:braid} For a subset $X$ of $\Bbb C,$ an $X$-braid 
$f:\Bbb I\to F_X$ is a path satisfying $f(0) = \op{id}_X$ 
and $f(1)\in \Si _X.$ 
In case $f(1)=\op{id}_X,$ $f$ is called a pure $X$-braid. 
$X$-braids $f$ and $g$ are equivalent, if $f\sim g.$ 
Let $[f]$ be the equivalence class containing $f$ with respect to $\sim $ 
for an $X$-braid $f.$ 
Let $A_X$ be the set of $X$-braids, 
$B_X = \{ [f] : f\in A_X\}$ and $P_X =\{ [f]: f\text{ is a pure }X\text{-braid }\}.$ 
For $f,g\in A_X,$ let $f\# g = f\cdot g^{f(1)}.$ 
\end{dfn}
In case $X$ is finite, the $X$-braid group is defined as the fundamental 
group of the orbit space $F_X/\Si _X$ in \cite{FoxNeuwirth:braid}. Then we can rewrite 
this definition in terms of paths in $F_X,$ for $F_X$ is a covering space 
of $F_X/\Si _X$ (see Remark 1.5 (2)). Analogously, we define the $X$-braid 
group $B_X$ for a finite or infinite set $X$ in the sense of the next 
proposition, where the subgroup $P_X$ of $B_X$ is 
just the fundamental group of a space $F_X.$ 
\begin{prop}
Let $[f][g] = [f\# g]$ for $f,g\in A_X.$ 
Then, this operation is well-defined and $B_X$ becomes a group. 
\end{prop}
\begin{pf}
 Since the operation is clearly well-defined, we show the 
existence of identity and inverse. The equivalence class of the constant map 
$[\op{id}_X]$ is clearly the left identity. 
For $f\in A_X,$ let 
$\si = f(1)^{-1}\in \Si _X.$ 
Then, $(f^-)^\si \in A_X$ and $[f][(f^-)^\si ]=[f\cdot (f^-)^{\si f(1)}]
=[f\cdot f^-]=e$ and hence 
$[(f^-)^\si ]$ is the left inverse of $[f].$ 
\end{pf}

To simplify the notation, we adopt some convention in set theory, i.e. 
a non-negative integer $n$ is the set $\{ 0,\cdot \cdot \cdot ,n-1\} ,$ $0$ is 
the empty set and $\om$ is the set of all non-negative integers. 
Hence, $m < n$ if and only if $m\in n$ for $m,n \in \om .$ 
We denote the set of positive integers by $\om _+.$ 
Then, an $n$-braid in the usual sense is the same as defined as above and 
we can see that the $X$-braid group $B_X$ coincides with the usual $n$-braid 
group if $X = n.$ 
We identify $f\in A_X$ with the indexed family of paths 
$(p_x:x\in X)$ in $\Bbb C$ 
satisfying $p_x(0)=x$ 
for $x\in X,$ $p_x(t) \neq p_y(t)$ for $x\neq y\in X$ 
and $\{ p_x(1): x\in X\} = X,$ 
i.e. $p_x(t) = f(t)(x).$ 
Then, the restriction $f|_Y$ for $Y\subset X$ is the subfamily 
$(p_x:x\in Y)$ of $f=(p_x:x\in X).$  
we call $p_x$ the $x$-string. 
Since a homotopy in a product space corresponds to 
homotopies in all components, we get the next proposition, which 
is obtained just by rewriting the definition of the equivalence $\sim.$  
\begin{prop} For $X$-braids $f = (p_x:x\in X), g= (q_x:x\in X),$ 
$f$ and $g$ are equivalent if and only if there exists an indexed family of 
homotopies $(H_x:x\in X)$ such that 
\begin{enumerate}
\item[(1)] $H_x:\Bbb I\times \Bbb I\to \Bbb C$ is continuous; 
\item[(2)] $H_x(s,t)\neq H_y(s,t)$ for $s,t\in \Bbb I, x\neq y$; 
\item[(3)] $H_x(0,t)=x, H_x(1,t)=p_x(1)=q_x(1)$ for $ t\in \Bbb I$;  
\item[(4)] $H_x(s,0) = p_x(s)$ and $H_x(s,1) = q_x(s)$ for $s\in \Bbb I.$
\end{enumerate}
\end{prop}
Now, we can see that $X$-braids are \lq\lq strongly 
isotopic\rq\rq in the sense of Artin \cite{Artin:braid} if and only if 
they are equivalent, 
in case $X$ is finite. 
\begin{dfn}\label{dfn:equivalent} Two $X$-braids 
$f$ and $g$ are strongly equivalent, if there exists a continuous map 
$H:\Bbb C\times \Bbb I\times \Bbb I \to \Bbb C$ such that 
\begin{enumerate}
\item[(1)] $H_{s,t}:\Bbb C\to \Bbb C$ is a homeomorphism, where 
$H_{s,t}(\al )=H(\al ,s,t)$; 
\item[(2)] $H(\al ,s,0) = H(\al ,0,t) = H(\al ,1,t) = \al $ for 
$\al \in \Bbb C, s,t\in \Bbb I$;
\item[(3)] $H(f(s)(x),s,1) = g(s)(x)$ for $s\in \Bbb I, x\in X.$
\end{enumerate}
\end{dfn}
In case $X$ is finite, E. Artin \cite[Theorem 6]{Artin:braid} showed that 
$X$-braids $f,g$ are equivalent if and only if they are 
strongly equivalent. More precisely, he showed that for given 
$H_x\; (x\in X)$ in Proposition 1.3 there exists $H$ in Definition 1.4 
such that $H(f(s)(x),s,t) = H_x(s,t)$ and $H(\al ,s,t) = \al$ if $|\al|$ 
is large enough. 
\begin{rmk}\label{rmk:limit}

\noindent
(1) If we induce the box product topology to the space $\Bbb C^\om ,$ 
the situation changes as follows. The path connected component of 
$\op{id}_\om$ consists of all $u\in F_\om$ such that 
$u(n) = n$ for all but finite $n\in \om .$ 
Therefore, the group obtained by the procedure in Definition~\ref{dfn:braid} and 
Proposition 1.2 in this case is the sum $\bigcup \{ B_n: n\in \om \}$ 
under the natural identification of $B_n$ as a subgroup of $B_{n+1}$ 
and the homomorphic image of the canonical 
map to $\Si _\om$ 
consists of permutations of finite support. 

\noindent
(2) One might suspect why we do not treat with the quotient space of 
$F_X$ using $\Si _X$ as in \cite{FoxNeuwirth:braid}. 
If we take such a quotient of $A_X$ in Definition~\ref{dfn:braid} in case that $X$ is 
infinite, the quotient space does not become Hausdorff nor the quotient map 
is a regular cover. These are the reasons. 
\end{rmk}
\section{Basic results of $B_X$ and $\om$-braids}
To show some results of $B_X$ the next proposition is necessary. For 
its proof we debt A. Tsuboi. 
\begin{prop} For any $X\subset \Bbb C,$ the space $F_X$ 
is path connected.
\end{prop} 
\begin{proof} Let $u \in F_X.$ 
We shall choose $(\al ,1/2)\in \Bbb I\times \Bbb I$ for each $x\in X$ so that the dogleg segments 
from $(x,0)$ to $(u(x),1)$ via $(\al ,1/2)$ do not intersect for distinct 
$x's.$ 
First, well-order $X$ so that $X = \{ x_\mu : \mu <\ka \}$ and 
the cardinality of $\{ \nu :\nu <\mu \}$ is less than $2^{\aleph _0}$ for each 
$\mu <\ka .$ 
Let $Q_\mu = (x_\mu ,0)$ and $S_\mu = (u(x_\mu ),1)$ for $\mu <\ka .$ 
Suppose that we have gotten 
$R_\nu \in \Bbb C\times \{ 1/2 \}\; (\nu <\mu )$ so that 
$Q_\nu R_\nu S_\nu $ do not intersect mutually for distinct $\nu$'s. 
Since there are $2^{\aleph _0}$-many planes 
which contain $Q_\mu $ and $S_\mu ,$ there exists a plane which 
contains $Q_\mu $ and $S_\mu $ but does neither contain any 
$Q_\nu $ nor $S_\nu $ ($\nu <\mu ).$ 
Take such a plane. Then, the intersection of the plane and 
$Q_\nu R_\nu S_\nu $ consists of at most two points for 
each $\nu <\mu .$ 
Therefore, there are less 
than $2^{\aleph _0}$-many points on the plane which are on some 
constructed $Q_\nu R_\nu S_\nu .$ 
Tracing $(\al ,1/2)$'s in the plane, we get $2^{\aleph _0}$-many 
distinct dogleg segments in the plane which connect $x_\mu$ and $u(x_\mu ).$ 
Therefore, we can choose $R_\mu $ so that 
$Q_\mu R_\mu S_\mu $ does not intersect with 
$Q_\nu R_\nu S_\nu $ for $\nu < \mu .$  
For $t\in \Bbb I,$ define $f(t)(x_\mu )$ so that 
$(f(t)(x_\mu ), t)$ is on $Q_\mu R_\mu S_\mu .$ 
Then, $f:\Bbb I\to F_X$ is a path from $\op{id}_X$ to $u.$ 
\end{proof}
\begin{thm} If $X$ and $Y$ have the same cardinality for 
$X,Y\subset \Bbb C,$ $B_X$ and $B_Y$ are isomorphic. 
\end{thm}
\begin{proof} Take a bijective function $\vp :X\to Y.$ 
Then, $\vp$ induces a homeomorphism between $F_X$ and $F_Y.$ 
Let $\si _f = f(1)$ for $f\in A_X \cup A_Y.$ 
For $f\in A_X$ and $g\in A_Y,$ 
define paths $\vp (f)$ in $F_Y$ and $\vp ^{-1}(g)$ in $F_Y$ by: 
$\vp (f)(t)(y) = f(t)(\vp ^{-1}(y))$ and 
$\vp ^{-1}(g)(t)(x) = g(t)(\vp (x))$ for $t\in \Bbb I, x\in X, y\in Y.$ 
Then $\vp (f)$ is a path from $\vp ^{-1}$ to $\si _f\vp ^{-1}.$ 
By Proposition 2.1, there exists a path $h$ from $\op{id}_Y$ 
to $\vp ^{-1}\in F_Y.$ Define $\Phi :A_X\to A_Y$ by: 
$\Phi (f) = h\cdot \vp (f)\cdot (h^-)^{\vp \si _f\vp ^{-1}}.$ 
It is easy to see $\Phi (f) \sim \Phi (g)$ for $f\sim g.$ 
On the other hand, $\vp ^{-1}(h\cdot \vp (f)\cdot (h^-)^{\vp \si _f\vp ^{-1}}) 
= \vp ^{-1}(h)\cdot f\cdot \vp ^{-1}((h^-)^{\vp \si _f\vp ^{-1}} 
= \vp ^{-1}(h)\cdot f\cdot (\vp ^{-1}(h^-))^{\si _f}.$ 
Therefore, $\vp ^{-1}(h^-)\cdot \vp ^{-1}(\Phi (f))\cdot (\vp ^{-1}(h))^{\si _f} \in A_X$ is equivalent to $f,$ which shows that 
$\Phi (f)\sim \Phi (g)$ implies $f\sim g .$ 
Therefore, $\Phi$ induces an injection from  $B_X$ to $B_Y.$ 
The surjectivity of the induced map can be proved similarly as above. 
The remaining thing to show is $\Phi (f\# g) \sim \Phi (f)\# \Phi (g).$ 
\begin{eqnarray*}
\Phi (f)\# \Phi (g) 
& = h\!\cdot\! \vp (f)\!\cdot\! (h^-)^{\vp \si _f\vp ^{-1}}\!\!\!\cdot\! 
  h^{\vp \si _f\vp ^{-1}}\!\!\!\cdot\! \vp (g)^{\vp \si _f\vp ^{-1}}\!\!\!
\cdot\! 
  (h^-)^{\vp \si _g\vp ^{-1}\vp \si _f\vp ^{-1}} \\
& \sim h\!\cdot\! \vp (f)\!\cdot\! \vp (g)^{\vp \si _f\vp ^{-1}}\!\cdot\! 
  (h^-)^{\vp \si _g\si _f\vp ^{-1}} \\
& = h\!\cdot\! \vp (f\# g)\!\cdot\! (h^-)^{\vp \si _g\si _f\vp ^{-1}} \\
& = h\!\cdot\! \vp (f\# g)\!\cdot\! (h^-)^{\vp \si _{f\# g}\vp ^{-1}} \\
& = \Phi (f\# g).
\end{eqnarray*}
\end{proof}
Next, we prove an exact sequence which shows a relationship between braid 
groups and permutation groups. 
\begin{prop}
The following exact sequence holds for any subset $X$ of $\Bbb C:$ 

$0\;\; \to\;\; P_X\;\;\to \;\; B_X\;\; \to \;\;\Si _X\;\;\to\;\;  0 \; .$
\end{prop}
\begin{proof} Let $h:B_X\to \Si _X$ be the canonical homomorphism, i.e. 
$\phantom{abcd}$ $h([f])(x)=f(1)(x)$ for $f\in A_X, x\in X.$ 
Then, $\op{Ker}(h) = P_X$ clearly. For $\si\in \Si _X$ there exists a path $f$ 
in $F_X$ from $\op{id}_X$ to $\si$ by Proposition 2.1. 
Then, $f\in A_X$ and $h([f]) = \si .$ 
\end{proof}
It is well-known that 
$B_n$ are torsion-free for $n\in \om $ [\bf 6\rm , \bf 3\rm ]. 
We shall show that $B_\om$ is torsion-free. It follows from the next theorem, 
which shows a fundamental property of $\om$-braids. 
\begin{thm}\label{thm:reduction} Let $f,g$ be $\om$-braids. Then, 
$f\sim g$ if and only 
if $f|_n\sim g|_n$ for the restrictions $f|_n, g|_n \; (n\in \om ).$ 
\end{thm}
\begin{cor}
The $\om$-braid group $B_\om $ is torsion-free.
\end{cor}
\begin{proof}
Let $[f]^m = e$ for $f\in A_\om , m\neq 0$ and $\si = f(1).$ Then, $\si ^m = e$ in $\Si _\om $ and hence there exists a partition 
$\{ E_n: n\in \om\}$ of $\om$ such that 
the cardinalities of $E_n$ are divisors of $m$ and each restriction of $\si$ 
to $E_n$ is a cyclic permutation on $E_n.$ The restriction $f|_{E_n}$ is an $E_n$-braid and $[f|_{E_n}] ^m = e.$ Therefore, $[f|_{E_n}] = e$ and hence 
$\si |_{E_n}=e$ for each $n.$ Now, we have shown that $f$ is a pure 
$\om$-braid. 

Since $[f|_n]^m = e$ by the assumption, 
$f|_n \sim \op{id}_n$ for every $n\in \om .$ Hence, $[f] = e$ by 
Theorem~\ref{thm:reduction}. 
\end{proof}

To show Theorem~\ref{thm:reduction}, some lemmas are necessary. 
\begin{lmm}
Let $f\in A_n$ be a pure $n$-braid for $n\in \om$ such that 
$[f] = e$ and $f(s)(i) = i$ for $i\in k \le n, s\in \Bbb I.$ 
Then, there exists a continuous map 
$H:\Bbb C\times \Bbb I\times \Bbb I \to \Bbb C$ such that 
\begin{eqnarray*}
(1) & H_{s,t}:\Bbb C\to \Bbb C \text{ is a homeomorphism, where }
H_{s,t}(\al )=H(\al ,s,t) ; \\ 
(2) & H(\al ,s,0)= \al \text{ and }H(j,s,1)=f(s)(j) \text{ for } j\in n;\\
(3) & H(\al ,0,t)= H(\al ,1,t) =\al ;\\
(4) & H(i,s,t)=i \text{ for }i\in k.
\end{eqnarray*}
\end{lmm}
\begin{proof}
It suffices to show the case $k=n\! -\! 1,$ since the conclusion can be obtained by 
repeated use of such a special case. 
By the assumption and \cite[Theorem 6]{Artin:braid}, 
$\op{id}_n$ and $f$ are strongly equivalent, that is, 
there exists a continuous map 
$H':\Bbb C\times \Bbb I\times \Bbb I \to \Bbb C$ 
satisfying: 
\begin{eqnarray*}
(1) & H'_{s,t}:\Bbb C \to \Bbb C \text{ is a homeomorphism }; \\
(2) & H'(\al ,s,0)= \al \text{ and }H'(j,s,1)=f(s)(j) 
\text{ for } j\in n;\\
(3) & H'(\al ,0,t)= H'(\al ,1,t)= \al ; \\
(4) & H'(m,s,t) = m \text{ for large enough }m.
\end{eqnarray*}
For $\ep > 0,$ define $\psi _\ep :\Bbb I\to [n-1,m]\; (\subset \Bbb C)$ as follows:  
$$\psi _\ep (s) = \left\{
\begin{array}{ll}
(1-\frac {s}{\ep })(n-1) + \frac {s}{\ep }m,& \text{for } 0\le s \le \ep ;\\
m, &\qquad \qquad \text{for } \ep < s \le 1-\ep ;\\
\frac {s-1+\ep }{\ep }(n-1) + (1-\frac {s-1+\ep }{\ep })m, &\quad \text{for }
1-\ep <s\le 1 
\end{array}
\right.
$$
Let $g_\ep (s) = H'(\psi _\ep (s),s,1).$ 
Since $H'(\al ,0,1) = H'(\al ,1,1) = \al$ for any $n\! -\! 1\le \al \le m$ and 
$H'(m,s,1) = m$ for any $s\in \Bbb I,$ 
there exists $\ep >0$ such that the real part of $g_\ep (s)$ is greater 
than $n-1-1/2.$ Fix such an $\ep .$ 
Using $H'((1-t)(n-1)+t\psi _\ep (s),s,1),$ we get a homotopy from 
$f$ to an $n$-braid whose $(n\!\! -\!\! 1)$-string $g_\ep $ varies in 
$\{ \al\in \Bbb C: {\mathbf Re}(\al )>n-1-1/2 \}$ keeping the $i$-strings 
($i\in n\!\! -\!\! 1$) fixed. Then, we can easily make 
the $(n\!\! -\!\! 1)$-string 
straight leaving the $i$-strings ($i\in n\!\! -\!\! 1$) fixed. 
To perform the works altogether, let $H''_i(s,t) = i$ for $i\in n\!\! -\!\! 1$ 
and 
$$
H''_{n-1}(s,t) = 
\left\{ 
\begin{array}{ll}
H'((1-2t)(n-1)+2t\psi _\ep (s),s,1) & \text{for }0\le t\le 1/2;\\
       (2-2t)g_\ep +(2t-1)(n-1) & \text{for }1/2< t\le 1.
\end{array}
\right. 
$$
To get the desired $H,$ let $H_i(s,t) = H''_i(s,1-t).$ 
Then, $H_i\; (i\in n)$ satisfy the following: 
\begin{eqnarray*}
(1) && H_i:\Bbb I\times \Bbb I \to \Bbb C \text{ is continuous and } 
H_i(s,t)\neq H_j(s,t) \text{ for }i\neq j; \\
(2) && H_i(s,0)= i \text{ and }H_i(s,1)=f(s)(i) \text{ for } i\in n;\\
(3) && H_i(0,t)= H_i(1,t) =i \text{ for }i\in n;\\
(4) && H_i(s,t)=i \text{ for }i\in n-1.
\end{eqnarray*}
Again by \cite[Theorem 6]{Artin:braid}, we can extend $H_i$'s to the desired $H.$  
\end{proof}
\begin{lmm}\label{lmm:spetial} Let $f$ be a pure $\om$-braid. 
Then, $f\sim \op{id}_\om $ if and only if $f|_n\sim \op{id}_n$ for every $n\in \om .$ 
\end{lmm}
\begin{proof} It suffices to show the one direction. Suppose that 
$f|_n\sim \op{id}_n$ for every $n\in \om .$ 
By induction we define pure $\om$-braids $f_m$ and a continuous map 
$H^m:\Bbb C\times\Bbb I\times\Bbb I\to \Bbb C$ as the following: 
\begin{eqnarray*}
(1)& f_0 = f, f_m(s)(n) = H^m(n,s,0), f_{m+1}(s)(n) = H^m(n,s,1); \\
(2)& H^m_{s,t}\text{ is an autohomeomorphism on }\Bbb C, \\
&\text{ where }H^m_{s,t}(\al ) = H^m(\al ,s,t,); \\
(3)& H^m(\al ,0,t) = H^m(\al ,1,t) = \al ; \\ 
(4)& H^m(k,s,t) = k \text{ for }k \in  m,  H^m(k,s,1) = k \text{ for }
k \in m+1.
\end{eqnarray*}
Using Lemma\ref{lmm:spetial}and the facts $[f|_m] = e\; (m\in \om )$ and 
$f_m(s)(k) = k \; (k \in m),$ we can define these above. 
Define $H:\Bbb I\times \Bbb I\to F_\om$ as follows: 
$$\left\{
\begin{array}{lll}
H(s,1)(n) &=& n; \\
H(s,t)(n) &=& H^m(n,s,2^{m+1}(t - \Si _{i=1}^m1/2^i)), \\ 
&&\text{ if }\; \Si _{i=1}^m1/2^i \le t \le \Si _{i=1}^{m+1}1/2^i. 
\end{array}
\right.
$$
Then, $H(0,t) = H(1,t) = \op{id}_\om ,$ $H(s,0) = f(s)$ and 
$H(s,1) = \op{id}_\om$ hold. The continuity of $H$ follows from the fact 
that $H^m(n,s,t) = n$ if $n<m$ and $\Si _{i=1}^m 1/2^i \le t.$ 
Therefore, $[f] = e.$ 
\end{proof}

{\it Proof of\/} Theorem~\ref{thm:reduction} 
It suffices to show the one direction. Suppose that 
$f|_n\sim g|_n$ for the restrictions $f|_n, g|_n \; (n\in \om ).$ 
Then, $(f\cdot g^-)|_n = f|_n\cdot g^-|_n \sim \op{id}_n$ for every $n\in \om .$ 
By Lemma 2.7, $f\cdot g^- \sim \op{id}_\om .$ 
Let $\si = f(1),$ then $g(1) = \si $ and $[g]^{-1} = [(g^-)^{\si ^{-1}}].$ 
Hence, $[f][g]^{-1} = [f\cdot (g^-)^{\si ^{-1}\si}] = [f\cdot g^-] 
= [\op{id}_\om ] = e,$ i.e. $f\sim g.$ 
\smallskip
\smallskip

It is known that the abelianization of $B_n,$ i.e. $B_n/(B_n)',$ is isomorphic 
to $\Bbb Z.$ 
In case of $B_\om $ some difference should occur, since $\Si _X$ 
coincides with its commutator subgroup $(\Si _X)'$ for an infinite $X.$ 
We don't know whether $B_\om = (B_\om )'$ or not. 
But, the abelianization of $B_\om$ is not isomorphic to $\Bbb Z,$ as we shall see in the following. 

Some definition is necessary to state the next theorem. 
An abelian group 
$A$ is called complete modulo the Ulm subgroup (abbreviated by 
\lq\lq complete mod-$U$\rq\rq ), if for any $x_n \in  A (0 < n\in \om )$ with $n!\mid  x_{n+1}-x_n$ there 
exists $x\in A$ such that $n!\mid  x-x_n$ for all $0<n \in \om.$ 
In other words, $A/U(A)$ is complete \cite{Fuchs:group},
where $U(A) = \bigcap _{n\in \om _+}n!A.$ 
Since any homomorphic image of a complete mod-$U$ abelian group is also 
complete mod-$U,$ 
a complete mod-$U$ abelian group has no summand isomorphic to $\Bbb Z.$ 
(See \cite{E:free} for further information about complete mod-$U$ groups.) 
This kind of group is related to the first integral singular homology groups 
of wild spaces \cite{E:union,E:free}. 
\begin{thm}
The abelianization of $B_\om$ is complete modulo the Ulm subgroup.  
\end{thm}
It is well-known that $\Si _X = (\Si _X)'$ for infinite $X$ 
\cite[p. 306]{Scott:group}. 
Therefore, we get the following by Proposition 2.3. 
\begin{lmm}
For an infinite $X\subset \Bbb C,$ $B_X = P_X(B_X)'.$ 
\end{lmm}
For $f\in A_X ,$ let $[f]_a$ be the element of the 
abelianization of $B_X$ corresponding to $f,$ i.e. the map $[f]\to [f]_a$ 
is the canonical homomorphism from $B_X$ to $B_X /(B_X)'.$ 
We shall treat with the cases $X = n$ for $n\in \om$ and $X = \om .$ 
\begin{lmm}
For any pure $\om$-braid $f$ and neighborhood $O$ of $\op{id}_\om $ in 
$F_\om ,$ there 
exists a pure $\om$-braid $g$ such that 
$[g]_a = [f]_a$ and $\op{Im}(g)\subset O.$ 
\end{lmm}
\begin{proof}
There exists $n\in \om$ such that $O$ depends on the $i$-th co-ordinates ($i\in n$). 
Reminding the proof of the fact $B_{n+1}/(B_{n+1})' \simeq \Bbb Z,$ we 
conclude the existence of a pure $(n\!\! +\!\! 1)$-braid $g'$ such that 
$[g']_a = [f|_{n+1}]_a$ and $g'(s)(i) = i$ for $i\in n$ and $s\in\Bbb I.$ 
Then, there exists an ($n$+1)-braid $h'$ such that $[h']\in (B_{n+1})'$ and 
$[g'] = [f|_{n+1}][h'].$ Since we may assume $|h'(s)(i)| \le n+1/2$ for 
$i\in n$ and $s\in\Bbb I,$ extend $h'$ to a pure  $\om$-braid $h$ so that 
$h(s)(i) = i$ for 
$n+1 \le i\in \om, s\in \Bbb I.$ 
Then, 
$[h]\in (B_\om)'$ and $(f\! \cdot \! h)|_{n+1} \sim g'.$ 
Take $H$ in Definition~\ref{dfn:equivalent} for $(f\cdot h)|_{n+1}$ and $g'$ and 
define $g$ by: $g(s)(i) = H((f\cdot h)(s)(i),s,1).$ 
We have gotten a pure $\om$-braid $g$ such that $g|_{n+1} = g'$ and 
$[g]=[f][h],$ which imply $g\in O$ and $[g]_a = [f]_a.$ 
\end{proof}

The next lemma is essentially included in \cite[Theorem 1.1]{E:union}. 
More precisely, what was necessary to prove it is that 
the one point union of 
cones $(CX,x)\vee (CY,y)$ satisfies the condition of the next lemma with 
the common point as base point and 
that the first integral 
singular homology group is 
an abelian group and 
a homomorphic image of the fundamental group. 
Therefore, we omit the proof. 
\begin{lmm}
Let $Y$ be a path-connected Hausdorff space and first countable at $y\in Y.$ 
Let $A$ be an abelian group which is a homomorphic image of $\pi _1(Y,y)$ and 
$h:\pi _1(Y,y)\to A$ be the homomorphism. 
Suppose that for any loop $f$ with base point 
$y$ there exist loops $f_n$ with base point $y$ $(n\in \om )$ 
such that $h([f_n]) = h([f])$ and $\op{Im}(f_n)$ converge to $y.$ 
Then, $A$ is complete mod-$U$. 
\end{lmm}
{\it Proof of\/} Theorem 2.8.
By Lemma 2.9, the abelianization of $B_\om $ is a homomorphic image 
of $P_\om .$ 
Now, the theorem follows from 
Lemmas 2.10 and 2.11. 
\smallskip
\smallskip
\begin{rmk} We have not succeeded to 
prove the torsion-freeness of $B_X $ for uncountable $X.$  
As we remarked before Theorem 2.8, we don't know whether $B_\om = (B_\om )'$ 
or not. 
\end{rmk}
\section{A representation of $B_\om$ as an automorphism group \\
on the unrestricted free product}
\bigskip
As is well-known, E. Artin \cite{Artin:braid} represented $B_n$ 
as an automorphism group 
on free groups of $n$-generators. On the other hand, G. Higman 
\cite{Higman:unrestrict} 
introduced a notion \lq\lq Unrestricted free product\rq\rq . In this section, 
we represent $B_\om$ as an automorphism group on 
the unrestricted free product of finitely generated free groups. 
Let $\Bbb Z_i\; (i\in \om )$ be copies of the integer group $\Bbb Z$ and 
$p^m_n:\ast _{i\in m}\Bbb Z_i\to \ast _{i\in n}\Bbb Z_i$ the canonical 
projection for $n\le m,$ where 
$\ast _{i\in m}\Bbb Z_i$ is the free product of $\Bbb Z_i$'s. 
We regard $\ast _{i\in m}\Bbb Z_i$ as the trivial 
group $\{ e\}$ in case $m = 0$ as usual. 
The unrestricted free product $\Cal G_\om$ is the inverse limit 
$\lim_{\leftarrow }(\ast_{i\in n} \Bbb Z_i,  p^m_n: n\le m, m,n\in \om ).$ 
Let $p_E:\Cal G_\om \to \ast _{i\in E}\Bbb Z_i$ be the induced projection for 
a finite subset $E$ of $\om.$  In the following, let $\de _i$ be the 
generator of $\Bbb Z_i$ which corresponds to $1$ in $\Bbb Z.$ 
For a group $G,$ we denote the automorphism group of $G$ by $\op{Aut}(G).$ 
\begin{thm} Fix an inverse system 
$(\ast_{i\in n} \Bbb Z_i,  p^m_n: n\le m, m,n\in \om )$ 
for $\Cal G_\om .$ 
Let ($\dagger $) be a property of $a\in \op{Aut}(\Cal G_\om )$ 
as the following: 
\begin{eqnarray}
(\dagger )\; & \text{For any } \; m\in \om \text{ there exist a finite subset }
\; E\; \text{ of }\; \om , h\; \text{and }\; \si \\ 
& \text{such that}\; h :\ast _{i\in m}\Bbb Z_i\to \ast _{i\in E}\Bbb Z_i\text{ is an isomorphism }, \si : m\to E  \\
& \text{is a bijection}, p_E a = h p_m\; , 
h(\de _i) = U_i^{-1}\de _{\si (i)}U_i \text{ for some } U_i\; , \\
& h(\Pi _{i\in m}\de _i ) = \Pi _{i\in E}\de _i\; ,
 \text{where the products are taken under the} \\
& \text{natural ordering on } \om.
\end{eqnarray}
Then, $B_\om$ is naturally isomorphic to the subgroup 
of $\op{Aut}(\Cal G_\om )$ 
consisting of all automorphisms satisfying ($\dagger$). 
\end{thm}
To show the theorem, some lemmas are necessary. 
\begin{lmm} \cite[Theorem 1]{Scott:group} 
Let $h:\Cal G_\om \to \Cal F$ be a homomorphism, where 
$\Cal F$ is a free group. 
Then, there exist $m\in \om $ and a homomorphism 
$\ov{h}:\ast _{i\in m}\Bbb Z_i \to \Cal F$ such that $h = \ov{h}p_m .$ 
\end{lmm}
\begin{lmm}
Suppose that $f_m:\Bbb I \to F_m\; (m\in \om _+)$ are satisfying 
$f_m(0) = \op{id}_m,$ $\{ f_m(1)(i):i\in m \} \subset \om$ and $f_m|_n \sim f_n$ for $n\le m, m,n\in \om _+.$ 
If $\bigcup _{m\in \om _+} \{ f_m(1)(i):i\in m \} = \om ,$ 
then there exists $f\in A_\om $ such that $f|_m \sim f_m$ for every 
$m\in \om _+.$ 
\end{lmm}
\begin{proof} By \cite[Theorem 6]{Artin:braid}, for each $m\in \om _+$ there exists 
$H^m:\Bbb C\times \Bbb I\times \Bbb I \to \Bbb C$ such that 
\begin{eqnarray*}
(1) & H^m_{s,t}:\Bbb C\to \Bbb C \text{ is a homeomorphism, where }
H^m_{s,t}(\al )=H^m(\al ,s,t) ; \\ 
(2) & H^m(\al ,s,0)= \al \text{ and }H^m(\al ,0,t)=H^m(\al ,1,t)=\al ; \\
(3) & H^m(f_{m+1}(s)(i),s,1)=f_m(s)(i) \text{ for }i\in m.
\end{eqnarray*}
Next, define $f^k_m :\Bbb I\to \Bbb C$ as follows: 
$f^0_m(s) = f_m(s)(m-1)$ and $f^k_m(s) = H^{m-k}(f^{k-1}_m(s),s,1)$ 
for $0 < k\in m.$ 
Finally, let $f(s)(m) = f^m_{m+1}(s)$ for $m\in \om .$ 
It is easy to check that $f(s) \in F_\om ,$ $f(1) = \op{id}_\om$ 
and $f:\Bbb I\to F_\om $ is continuous. 
By definition, $f|_m \sim f_m$ for every $m\in \om _+.$  
Now, $\bigcup _{m\in \om _+} \{ f_m(1)(i):i\in m \} = \om $ implies 
$f(1) \in \Si _\om$ and we have shown the lemma. 
\end{proof}

{\it Proof of} Theorem 3.1. 
For $f\in A_\om ,$ let $\si = f(1)$ and $n\in \om _+.$ 
The restriction $f|_n$ does not always belong to $B_n,$ but is still a braid 
in the sense of \cite{Artin:braid}. Therefore, 
$f|_n$ induces an isomorphism 
$e_n^f:\ast _{i\in n}\Bbb Z_i\to \ast _{i\in n}\Bbb Z_{\si (i)}.$ 
For $n\in \om ,$ denote the set $\{ \si (i):i\in n\}$ by $\si [n].$ 
Then, the following diagram 3.4 commutes, i.e. $e_n^fp_n^m = p_{\si [n]}^{\si [m]}e_m^f$ 
for $n\le m\; (m,n\in \om ).$ 

$${\rm{Diagram 3.4}}\quad 
\xymatrix{
\ast _{i\in m}\Bbb Z_i  \ar[d]_{p_n^m} \ar[r]^{{e_m^f}} & 
\ast _{i\in \si [m]}\Bbb Z_i \ar[d]^{p_{\si [n]}^{\si [m]}}\\
\ast _{i\in n}\Bbb Z_i
\ar[r]_{e_n^f} & \ast _{i\in \si [n]}\Bbb Z_i}
$$

Since $\si [n] \subset \si [m] $
for $n\le m\; (m,n\in \om )$ and $\om = \bigcup _{n\in \om }\si [n],$ 
$\Cal G_\om = \lim_{\leftarrow }
(\ast_{i\in \si [n]} \Bbb Z_i,  p^{\si [m]}_{\si [n]}: n\le m, m,n\in \om ).$ Therefore, $f$ induces an automorphism $a_f$ on $\Cal G_\om ,$ 
which satisfies the property ($\dagger$) by \cite[Theorem 15]{Artin:braid}. 
It is easy to see that $a_f$ only depends on $[f]$ by Theorem 2.4 and 
the map $[f] \to a_f$ is a homomorphism from $B_\om$ to 
$\op{Aut}(\Cal G_\om ).$ The injectivity of this homomorphism follows again 
from Theorem 2.4. 

To show the surjectivity of this homomorphism, let $a\in \op{Aut}(\Cal G_\om )$ satisfy the property ($\dagger$). 
Then, there exist $h_m$ and a finite subset $E_m$ of $\om $ which satisfy the properties in ($\dagger$) for each $m\in \om .$ 
First, we show $E_n \subset E_m$ for $n\le m.$ 
For $i\in E_n,$ there is $g_i\in \Cal G_\om$ such that $a(g_i)=\de _i,$ 
where we identify $\Bbb Z_i$ with the corresponding subgroup 
in $\Cal G_\om .$  
Then, $\de _i = p_{E_n}a(g_i) = h_np_n(g_i)$ and hence $p_n(g_i)\neq e,$ which implies $p_m(g_i)\neq e.$ 
Therefore, $p_{E_m}(\de _i) = p_{E_m}a(g_i) = h_mp_m(g_i) \neq e,$ which implies $i\in E_m.$ 
For $i\in m,$ $p_{E_m}a(\de _i) = h_mp_m(\de_i) = h_m(\de _i)$ and hence 
$p_{E_n}^{E_m}h_m(\de _i) = p_{E_n}^{E_m}p_{E_m}a(\de _i) = 
p_{E_n}a(\de _i) = h_np_n(\de _i) = h_np_n^mp_m(\de _i) = h_np_n^m(\de _i).$ 
Therefore, $p_{E_n}^{E_m}h_m = h_np_n^m,$ i.e. the following diagram commutes. 
$$\rm{Diagram} 3.5\quad 
\xymatrix{
\ast _{i\in m}\Bbb Z_i  \ar[d]_{p_n^m} \ar[r]^{h_m} & 
\ast _{i\in E_m}\Bbb Z_i \ar[d]^{p_{E_n}^{E_m}}\\
\ast _{i\in n}\Bbb Z_i
\ar[r]_{h_n} & \ast _{i\in E_n}\Bbb Z_i}
$$
Since $h_m$ satisfies the conditions in ($\dagger$) and Diagram 3.5, by 
\cite[Theorem 16]{Artin:braid} 
there exists a path $f_m:\Bbb I\to F_m$ such that $f_m(0) = \op{id}_m, $ 
$\{ f_m(1)(i):i\in m\} = E_m,$
$e^m_{f_m} = h_m$ and $f_m|n\sim f_n$ for $n\le m, m,n\in \om _+.$  
If we can show that $f_m \; (m\in \om _+)$ satisfy the conditions of 
Lemma 3.3, we get $f\in A_\om $ so that $a = a_f.$ Therefore, it suffices 
to show $\bigcup _{m\in \om _+}E_m = \om.$ 
For any $j\in \om ,$ there exist $m\in \om _+$ and 
$\ov{h}:\ast_{i\in m}\Bbb Z_i\to \Bbb Z_j$ with $p_{\{ j\}}a = \ov{h}p_m$ 
by Lemma 3.2. Take $g$ so that $a(g) = \de _j.$ 
Then, $\ov{h}p_m(g) = p_{\{ j\}}a(g) \neq e$ and hence $p_m(g)\neq e.$ 
Since $h_m$ is an isomorphism, 
$p_{E_m}(\de _j) = p_{E_m}a(g) = h_mp_m(g) \neq e,$ which implies 
$j\in E_m.$ 
\smallskip
\smallskip
\begin{rmk}
Instead of $\Cal G_\om ,$ there is another candidate to represent $B_\om $ 
by its 
automorphism group. It is the fundamental group of the so-called Hawaiian 
ear ring, which is a subgroup of $\Cal G_\om $ and studied in \cite{E:free}. 
But, a natural $\om $-braid has no naturally corresponding 
automorphism on the fundamental group of the Hawaiian ear ring. 
For instance, think of a pure $\om$-braid such that the first string goes 
straight, but the others go around the first string. 
\end{rmk}

\medskip
{\large Supplementary remark}
In the case of a finite subset $X$ of $\Bbb C,$ regarding $\Bbb C^X$ as the mapping space 
from the discrete space $X$ to $\Bbb C,$ we can see natural correspondences among 
 $f\in A_X,$ a continuous maps $f':X\times \Bbb I\to \Bbb C$  
$( f'(x,t) = f(t)(x))$ and a level preserving embedding 
$f'':X\times \Bbb I \to \Bbb C\times \Bbb I$  $( f''(x,t) = (f(t)(x),t))$ 
whose image $f''(X\times \Bbb I) \subset \Bbb C\times \Bbb I$ coincides 
with a finite braid in the intuitive sense. For $f,g\in A_X,$ a homotopy 
between $f$ and $g$ in $F_X$ relative to $\{ 0,1\}$ corresponds to a level-
preserving (ambient, in fact) isotopy between 
the two embeddings $f''$ and $g''$ keeping $X\times \{ 0,1\}$ fixed. 
In case $X$ is infinite, 
$f''$ may fail to be an embedding and is just a level preserving continuous injective map in general. And $f\sim g$ if and only if there exists a level preserving homotopy $H_t:X\times \Bbb I \to \Bbb C\times \Bbb I\; (t\in \Bbb I)$ such that $H_0= f'',$ $H_1 = g''$ and $H_t$ is injective. This is our basic view point to infinitary braid groups $B_X$ in this paper. 

\providecommand{\bysame}{\leavevmode\hbox to3em{\hrulefill}\thinspace}

\end{document}